\documentclass[a4paper,11pt,reqno]{amsart}
\usepackage{amsmath,amsfonts,amsthm,amssymb,color}
\usepackage[T1]{fontenc}
\usepackage{pdfsync}
\usepackage{csquotes}
\usepackage{graphicx}
\usepackage{pstricks}
\usepackage{lmodern}

\newcommand{\id}{\mbox{Id}}

\newcommand{\1}{{\bf 1}}

\newcommand{\bu}{\mathbf{U}}

\newcommand{\C}{\mathbb C}

\newcommand{\R}{\mathbb R}

\newcommand{\ca}{\mathcal A}

\newcommand{\cac}{\mathcal C}

\newcommand{\cl}{\mathcal L}

\newcommand{\cn}{\mathcal N}

\newcommand{\al}{\alpha}
\newcommand{\der}{\delta}

\newcommand{\ga}{\gamma}
\newcommand{\ka}{\kappa}
\newcommand{\la}{\lambda}

\newcommand{\si}{\sigma}

\newcommand{\vp}{\varphi}

\newcommand{\lln}{\left|}
\newcommand{\rrn}{\right|}

\newtheorem{theorem}{Theorem}[section]

\newtheorem{definition}[theorem]{Definition}

\newtheorem{lemma}[theorem]{Lemma}

\newtheorem{proposition}[theorem]{Proposition}

\theoremstyle{remark}
\newtheorem{remark}[theorem]{Remark}

\date{\today}

\begin{document}

\makeatletter
\def\@settitle{\begin{center}%
  \baselineskip14\p@\relax
    \normalfont\LARGE
\@title
  \end{center}%
}
\makeatother

\title{Integration with respect to the Hermitian fractional Brownian motion}

\author{Aur\'elien Deya}
\address[A. Deya]{Institut Elie Cartan, University of Lorraine
B.P. 239, 54506 Vandoeuvre-l\`es-Nancy, Cedex
France}
\email{aurelien.deya@univ-lorraine.fr}

\keywords{Hermitian fractional Brownian motion; integration theory; pathwise approach; non-commutative stochastic calculus; non-commutative fractional Brownian motion}

\subjclass[2010]{15B52,60G22,60H05,46L53}

\begin{abstract}
For every $d\geq 1$, we consider the $d$-dimensional Hermitian fractional Brownian motion (HfBm), that is the process with values in the space of $(d\times d)$-Hermitian matrices and with upper-diagonal entries given by complex fractional Brownian motions of Hurst index $H\in (0,1)$. 

\smallskip

We follow the approach of [A. Deya and R. Schott: \textit{On the rough paths approach to non-commutative stochastic calculus}, JFA (2013)] to define a natural integral with respect to the HfBm when $H>\frac13$, and identify this interpretation with the rough integral with respect to the $d^2$ entries of the matrix. Using this correspondence, we establish a convenient Itô--Stratonovich formula for the Hermitian Brownian motion. 

\smallskip

Finally, we show that at least when $H\geq \frac12$, and as the size $d$ of the matrix tends to infinity, the integral with respect to the HfBm converges (in the tracial sense) to the integral with respect to the so-called non-commutative fractional Brownian motion. 

\end{abstract} 

\maketitle

\section{Introduction}

We propose to investigate some integration issues related to the so-called \textit{Hermitian fractional Brownian motion}, that is the fractional extension of Dyson's celebrated Hermitian Brownian motion \cite{dyson}. The specific definition of the Hermitian fractional Brownian motion (HfBm in the sequel) naturally goes as follows. For some fixed parameter $H\in (0,1)$, consider first two independent families  $(x(i,j))_{i\geq j\geq 1}$ and $(\tilde{x}(i,j))_{i\geq j\geq 1}$ of independent fractional Brownian motions with common Hurst index $H$, defined on  a classical probability space $(\Omega,\mathcal{F},\mathbb{P})$. Then, for every fixed (finite) dimension $d\geq 1$, we define the ($d$-dimensional) HfBm of Hurst index $H$ as the process $X^{(d)}$ with values in the space of the $(d\times d)$-Hermitian matrices and with upper-diagonal entries given for every $t\geq 0$ by
\begin{equation}\label{defi-hfbm}
\begin{split}
&X^{(d)}_t(i,j):=\frac{1}{\sqrt{2d}} \big(x_t(i,j)+\imath\, \tilde{x}_t(i,j)\big) \quad \text{for} \ 1\leq j<i\leq d \ ,\\
&   X^{(d)}_t(i,i):=\frac{x_t(i,i)}{\sqrt{d}} \quad \text{for} \ 1\leq i\leq d \ .
\end{split}
\end{equation}
Observe that the classical Hermitian Brownian motion is then nothing but the HfBm of Hurst index $H=\frac12$. 

\

The HfBm (or more precisely its direct counterpart in the space of symmetric matrices) was already at the core of the analysis of \cite{nualart-perez-abreu,matrix-approx}, through the consideration of the associated Dyson process (that is, the process derived from the eigenvalues of $X^{(d)}_t$) and the stochastic dynamics governing it. We will here follow a slightly different direction and rather focus on integration with respect to $X^{(d)}$ itself, seen as a process with values in a non-commutative algebra. In fact, our objectives can essentially be summarized along two (related) lines of research:

\

\noindent
$(i)$ First, and in the continuation of \cite{deya-schott,deya-schott-2,deya-schott-3}, we propose to develop a pathwise approach to integration with respect to $X^{(d)}$, that is a pathwise way to interpret the integral $\int_0^1 A_u \mathrm{d} X^{(d)}_u B_u$, for $A,B$ in a suitable class of matrix-valued processes, and where $A_u\mathrm{d} X^{(d)}_u B_u$ is simply understood as a product of $(d\times d)$-matrices. Note that as soon as $H\neq \frac12$, the entries of $X^{(d)}$ no longer satisfy the martingale property, so that the integral $\int_0^1 A_u \mathrm{d}X^{(d)}_u B_u$ cannot be (componentwise) interpreted in the classical Itô sense anymore. We will overcome this difficulty by following the developments of \cite{deya-schott} on rough pathwise integration in a general algebra, which will at least cover the situation where $H>\frac13$. Our aim here is also to point out the fact that the resulting construction coincides with the rough-path interpretation of the integral with respect to the $d^2$-dimensional process $\{X^{(d)}(i,j)\}_{1\leq i,j\leq d}$, which allows us to make a link between the considerations of \cite{deya-schott} and the more classical rough-paths approach to finite-dimensional integration (as displayed in \cite{gubi} or more recently in \cite{friz-hairer}). As an illustration of the possibilities offered by the pathwise approach, we will finally exhibit a clear Itô--Stratonovich conversion formula for the Hermitian Brownian motion (see Proposition \ref{prop:ito-strato} below). 

\

\noindent
$(ii)$ Then, in the framework of non-commutative probability theory and at least when $H\geq \frac12$, we intend to emphasize the relevance of the HfBm as a matrix model (or a matrix approximation) for the so-called \textit{non-commutative fractional Brownian motion} (NC-fBm in the sequel). Let us recall that the NC-fBm has first been introduced in \cite{nourdin-taqqu} as a natural fractional extension of the celebrated free Brownian motion, and then further studied in \cite{deya-schott-3,nourdin-fbm,matrix-approx}. In Section \ref{sec:asympt-results} below, we will exhibit a convergence result (as the dimension $d$ goes to infinity) at the level of the processes themselves, but also at the level of the stochastic integrals these processes generate, which will both illustrate the robustness of the approximation and the consistency of the stochastic integrals. The convergence will therein be interpreted in the tracial sense, and the result thus gives a thorough account on the asymptotic behaviour of the mean spectral distribution of the processes under consideration (either $X^{(d)}$ or the integrals it generates).

\

The study is naturally organized along the above two-part splitting: in Section \ref{sec:integr-hfbm}, we focus on integration with respect to the HfBm for some fixed \emph{finite} dimension $d\geq 1$, while in Section \ref{sec:asympt-results}, we examine the limit of these objects (from a spectral perspective) as $d$ goes to infinity, and make the link with the non-commutative fractional Brownian motion.

\

Throughout the paper, we will denote the increments of any vector-valued path $(g_t)_{t\geq 0}$ by $\der g_{st}:=g_t-g_s$, for all $s,t\geq 0$.

\section{Integration with respect to HfBm}\label{sec:integr-hfbm}

Our first objective is to provide a clear interpretation of the integral against the HfBm $X^{(d)}$, for some fixed dimension $d\geq 1$. To be more specific, we are interested in the interpretation of the product model $\int A_u \mathrm{d}X^{(d)}_u B_u$, for processes $A,B$ taking values in a class of $(d\times d)$-matrices to be determined.

\smallskip

To this end, we propose to adapt the developments of \cite[Section 4]{deya-schott} (about rough integration in a general algebra) to the setting under consideration, that is to the algebra $\ca^{(d)}:=\C^{d,d}$ and the driving process $X^{(d)}$, along an almost sure formulation (due to the deterministic framework of \cite[Section 4]{deya-schott}). 

\smallskip

For this adaptation to be possible, \textit{we need to assume, throughout the section, that $X^{(d)}$ is a HfBm of Hurst index $H>\frac13$. Also, we fix $\frac13 < \ga < H$, and recall that in this case, $X^{(d)}$ is a $\ga$-Hölder process (a.s.).}

\subsection{Rough-path approach to integration with respect to $X^{(d)}$}\label{sec:general-rp}

\

\smallskip

We denote by $(E_{ij})_{1\leq i,j\leq d}$ the canonical basis of $\ca^{(d)}$, and consider the norm $\|U\|^2:=\sum_{i,j=1}^d |U(i,j)|^2$ for every $U\in \ca^{(d)}$. In the same way, we consider, for all $\bu \in (\ca^{(d)})^{\otimes 2}$ and $\mathcal{U}\in (\ca^{(d)})^{\otimes 3}$, the standard norms
$$\|\bu\|^2:=\sum_{i,j,k,\ell=1}^d |\bu((i,j),(k,\ell))|^2 \quad \text{and} \quad \|\mathcal{U}\|^2:=\sum_{i,j,k,\ell,m,n=1}^d |\mathcal{U}((i,j),(k,\ell),(m,n))|^2\ ,$$
where $\bu((i,j),(k,\ell))$ and $\mathcal{U}((i,j),(k,\ell),(m,n))$ refer of course to the coordinates of $\bu$ and $\mathcal{U}$ in the canonical bases $(E_{ij}\otimes E_{k\ell})$ and $(E_{ij}\otimes E_{k\ell}\otimes E_{mn})$.

\smallskip

The product interactions between $\ca^{(d)}$, $(\ca^{(d)})^{\otimes 2}$ and $ (\ca^{(d)})^{\otimes 3}$ will all be denoted by $\sharp$. To be more specific, we define the operation $\sharp$ as the linear extension of 
$$(U_1 \otimes U_2) \sharp Y=Y\sharp (U_1 \otimes U_2):=U_1Y  U_2\ , \quad  U_1,U_2,Y \in \ca^{(d)} \ ,$$
or as the linear extension of
$$Y \sharp (U_1 \otimes U_2 \otimes U_3) :=(U_1Y U_2) \otimes U_3 \quad , \quad (U_1 \otimes U_2 \otimes U_3) \sharp Y:=U_1 \otimes (U_2YU_3) \ .$$

\subsubsection{Product Lévy area}\label{subsec:tensor}

The following central object appears in \cite[Section 4]{deya-schott} as a natural \enquote{product} version of the classical Lévy area at the core of rough paths theory:

\begin{definition}\label{defi:aire-gene}
We call \emph{product Lévy area} above $X^{(d)}$ any process $\{\mathbf{X}_{st}\}_{0\leq s\leq t\leq 1}$ with values in $\cl(\ca^{(d)}\otimes \ca^{(d)},\ca^{(d)})$ such that, almost surely:

\smallskip

\noindent
(i) ($2\ga$-roughness) There exists a constant $c>0$ such that for all $0\leq s\leq t\leq 1$ and $\mathbf{U}\in \ca^{(d)}\otimes\ca^{(d)}$,
\begin{equation}\label{regu-levy-area}
\| \mathbf{X}_{st}[ \mathbf{U}] \| \leq c \, |t-s|^{2\ga} \|\mathbf{U}\| \ .
\end{equation}

\smallskip

\noindent
(ii) (Product Chen identity) For all $0\leq s\leq u\leq t\leq 1$ and $\mathbf{U}\in \ca^{(d)}\otimes\ca^{(d)}$, 
\begin{equation}\label{chen}
\mathbf{X}_{st}[\mathbf{U}]-\mathbf{X}_{su}[\mathbf{U}]-\mathbf{X}_{ut}[\mathbf{U}]=(\mathbf{U} \sharp \der X^{(d)}_{su}) \, \der X^{(d)}_{ut}\ .
\end{equation}
\end{definition}

\

In the finite-dimensional setting that we consider here, there is in fact a one-to-one correspondence between the set of product Lévy areas above $X^{(d)}$ and the set of (classical) Lévy areas above the $d^2$-dimensional process $(X^{(d)}(i,j))_{1\leq i,j\leq d}$. Let us recall here that, along the standard terminology of rough paths theory, a (classical) Lévy area above $(X^{(d)}(i,j))_{1\leq i,j\leq d}$ is a two-parameter path $(\mathbf{X}^{\mathbf{2}}_{st})_{s,t\in [0,1]}$ with values in  $(\ca^{(d)})^{\otimes 2}$ such that for all $0\leq s\leq u\leq t\leq 1$ and $1\leq i,j,k,\ell\leq d$, $\| \mathbf{X}^{\mathbf{2}}_{st}\| \leq c\, |t-s|^{2\ga}$ and 
$$\mathbf{X}^{\mathbf{2}}_{st}((i,j),(k,\ell))-\mathbf{X}^{\mathbf{2}}_{su}((i,j),(k,\ell))-\mathbf{X}^{\mathbf{2}}_{ut}((i,j),(k,\ell))=\der X^{(d)}_{su}(i,j) \der X^{(d)}_{ut}(k,\ell) \ .$$

The following (readily-checked) property thus makes a first link between the algebra approach of \cite[Section 4]{deya-schott} and the standard finite-dimensional rough-path formalism:

\begin{lemma}
There is a one-to-one relation $\mathbf{X}^{\mathbf{2}} \mapsto \mathbf{X}$ between the set of (classical) Lévy areas above the $d^2$-dimensional process $(X^{(d)}(i,j))_{1\leq i,j\leq d}$ and the set of product Lévy areas above $X^{(d)}$, given by the formula: for all $0\leq s\leq t\leq 1$, $1\leq i,j\leq d$ and $U,V\in \ca^{(d)}$,
\begin{equation}\label{correspondence-levy-areas}
\mathbf{X}_{st}\big[U\otimes V\big](i,j):=\sum_{k,\ell_1,\ell_2=1}^d U(i,k)V(\ell_1,\ell_2) \mathbf{X}^{\mathbf{2}}_{st}((k,\ell_1),(\ell_2,j)) \ .
\end{equation}
\end{lemma}

\subsubsection{Controlled biprocesses and integration}\label{subsec:main}

Following again the ideas of \cite{deya-schott}, let us now turn to the presentation of the class of integrands we shall focus on. As usual, a few topological considerations need to be introduced first. For $V:=(\ca^{(d)})^{\otimes n}$ ($n=1,2,3$), we denote by $\cac_1([0,1];V)$ the set of continuous $V$-valued maps on $[0,1]$, and by $\cac_2([0,1];V)$ the set of continuous $V$-valued maps on the simplex $\{0\leq s\leq t\leq 1\}$ that vanish on the diagonal. Then for every $\al\in (0,1)$, we define the $\al$-Hölder space $\cac_1^\al([0,1];V)$, resp. $\cac_2^\al([0,1];V)$, as the subset of paths $h\in \cac_1([0,1];V)$, resp. $h\in \cac_2([0,1];V)$, for which the following seminorm is finite:
$$\cn[h;\cac_1^\al([0,1];V)]:= \sup_{0\leq s<t \leq 1} \frac{\|\der h_{st}\|}{\lln t-s \rrn^\al}  \quad , \quad \text{resp.} \quad \cn[h;\cac_2^\al([0,1];V)]:=\sup_{0\leq s<t \leq 1} \frac{\| h_{st}\|}{\lln t-s \rrn^\al} \ .$$

\begin{definition}\label{def:control-biproc}
We call \emph{controlled biprocess} on $[0,1]$ any process $\bu \in \cac_1^\ga([0,1];(\ca^{(d)})^{\otimes 2})$ whose increments can be expanded as
\begin{equation}\label{decompo-bipro}
(\der \bu)_{st}=(\der X^{(d)})_{st} \sharp \mathcal{U}_s^{X,1} +\mathcal{U}_s^{X,2} \sharp (\der X^{(d)})_{st}+ \bu^\flat_{st} \ , \quad 0\leq s\leq t\leq 1  \ ,
\end{equation}
for some processes $\mathcal{U}^{X,1},\mathcal{U}^{X,2}\in \cac_1^{\ga}([0,1];(\ca^{(d)})^{\otimes 3})$ and $\bu^\flat \in \cac_2^{2\ga}([0,1];(\ca^{(d)})^{\otimes 2})$. In the sequel, we denote by $\mathbf{Q}$ the space of controlled biprocesses on $[0,1]$.
\end{definition}

\smallskip

Of course, the conditions in the above definition must all be understood in an almost-sure sense. A basic example of such a controlled biprocess is provided by the path $\bu_t:=P(X^{(d)}_t) \otimes Q( X^{(d)}_t)$, for fixed polynomials $P,Q$. It is indeed easy to check that the increments of $\bu$ can be expanded as in (\ref{decompo-bipro}), with
$$\mathcal{U}^{X,1}_s := \partial P\big(X^{(d)}_s\big)  \otimes Q\big(X^{(d)}_s\big) \quad , \quad \mathcal{U}^{X,2}_s:=P\big( X^{(d)}_s\big) \otimes \partial Q\big(X^{(d)}_s\big) \ ,$$
where, in this algebra setting, we define the derivative $\partial P(X)$ as the linear extension of the formula $\partial X^m=\sum_{i=0}^{m-1} X^i \otimes X^{m-1-i}$. More examples of controlled biprocesses, related to the so-called controlled processes, can be derived from \cite[Proposition 4.10]{deya-schott}.

\

In this finite-dimensional setting, controlled biprocesses happen to be a particular case of $X^{(d)}$-controlled path (in the sense of Gubinelli \cite{gubi}), which allows us to go ahead with the analogy between \cite{deya-schott} and standard rough paths theory:

\begin{lemma}\label{lem:link-controlled-paths}
Let $\bu \in \mathbf{Q}$ with decomposition (\ref{decompo-bipro}). Then for all fixed $1\leq i,j\leq d$, the ($d^2$-dimensional) path $(\bu((i,k),(\ell,j)))_{1\leq k,\ell\leq d}$ is controlled (in the classical sense of \cite[Definition 4.6]{friz-hairer}) with respect to $(X(m,n))_{1\leq m,n\leq d}$, with Gubinelli derivative given for all $s\in [0,1]$ and $1\leq k,\ell,m,n\leq d$ by 
$$\bu_s'((i,j);(k,\ell),(m,n)):=\mathcal{U}^{X,1}_s((i,m),(n,k),(\ell,j))+\mathcal{U}^{X,2}_s((i,k),(\ell,m),(n,j)) \ .$$
\end{lemma}

\smallskip

Given a product Lévy area $\mathbf{X}$ above $X^{(d)}$, we denote the \enquote{dual} of $\mathbf{X}$ as $\mathbf{X}^{\ast}$, that is 
$$\mathbf{X}^{\ast}_{st}[U\otimes V]:=\mathbf{X}_{st}[V^\ast \otimes U^\ast]^\ast \quad , \quad  \text{for all} \  U,V\in \ca^{(d)}\ .$$
Our interpretation of the integral against $X^{(d)}$ can now be read as follows (as an application of \cite[Proposition 4.12]{deya-schott}):
\begin{proposition}\label{prop-int-gen}
Let $\mathbf{X}$ be a product Lévy area above $X^{(d)}$. Then for every $\bu\in \mathbf{Q}$ with decomposition (\ref{decompo-bipro}), all $0\leq s\leq t\leq 1$ and every subdivision $D_{st} = \{t_0=s<t_1 <\ldots<t_n=t\}$ of $[s,t]$ with mesh $|D_{st}|$ tending to $0$, the corrected Riemann sum
\begin{equation}\label{corrected-riemmann-sums}
\sum_{t_i\in D_{st}} \Big\{ \bu_{t_i} \sharp (\der X^{(d)})_{t_it_{i+1}}+[\mathbf{X}_{t_it_{i+1}}\times \id](\mathcal{U}^{X,1}_{t_i})+[\id \times \mathbf{X}^{\ast}_{t_it_{i+1}}](\mathcal{U}^{X,2}_{t_i})\Big\}
\end{equation}
converges almost surely in $\ca^{(d)}$ as $|D_{st}| \to 0$. We call the limit the \emph{rough integral} (from $s$ to $t$) of $\bu$ against $\mathbb{X}:=(X^{(d)},\mathbf{X})$, and denote it by $\int_s^t \bu_u \sharp \mathrm{d}\mathbb{X}_u$. 
\end{proposition}

Using straightforward pathwise expansions, this interpretation can again be related to more standard rough constructions: 

\begin{lemma}\label{lem:consistency}
Assume that we are given a product Lévy area $\mathbf{X}$ above $X^{(d)}$ and consider $\mathbb{X}:=(X^{(d)},\mathbf{X})$. Then for all $\bu\in \mathbf{Q}$ and $1\leq i,j\leq d$, one has almost surely
\begin{equation}\label{consist-gene}
\Big(\int_0^1 \bu_u \sharp \mathrm{d}\mathbb{X}_u\Big)(i,j)=\int_0^1\sum_{k,\ell=1}^d  \bu_u((i,k),(\ell,j))\mathrm{d}X^{(d)}_u(k,\ell)  \ ,
\end{equation}
where the latter integral is interpreted as the rough integral (in the sense of \cite[Theorem 4.10]{friz-hairer}) of the controlled path $(\bu((i,k),(\ell,j)))_{1\leq k,\ell\leq d}$
(along Lemma \ref{lem:link-controlled-paths}), considering the (classical) L{\'e}vy area derived from $\mathbf{X}$ through relation (\ref{correspondence-levy-areas}). 
\end{lemma}

\begin{remark}\label{rk:convergence-l-p}
We have chosen to express these integration results in an almost-sure way, but, using again the considerations of \cite[Section 4]{deya-schott} and under suitable moment conditions on the integrands, this rough approach could also easily be formulated in some $L^p(\Omega)$-sense. For instance, if we assume that the random constant $c$ in (\ref{regu-levy-area}) admits finite moments of any order, and if the expansion (\ref{decompo-bipro}) of the integrand $\bu$ is such that
$$\mathbb{E}\big[ \big|\cn[\mathcal{U}^{X,i};\cac_1^{\ga}([0,1];(\ca^{(d)})^{\otimes 3})]\big|^r \big] < \infty \quad \text{and} \quad  \mathbb{E}\big[ \big|\cn[\bu^\flat ; \cac_2^{2\ga}([0,1];(\ca^{(d)})^{\otimes 2})]\big|^r \big] < \infty$$
for every $r\geq 1$, then the convergence in Proposition (\ref{prop-int-gen}) holds true in $L^p(\Omega)$ as well, for every $p\geq 1$. This is in fact a direct consequence of the control derived from the so-called sewing map (\cite[Theorem 4.2]{deya-schott}).
\end{remark}

\begin{remark}
The above Definition \ref{defi:aire-gene}, Definition \ref{def:control-biproc} and Proposition \ref{prop-int-gen} are thus essentially borrowed from \cite[Section 4]{deya-schott}, where integration with respect to a Hölder driver $X$ in a general algebra $\ca$ is considered. It is worth mentioning however that additional \enquote{adaptedness} conditions arise in \cite{deya-schott}: for instance, it is therein assumed that for every $t\in [0,1]$, $\bu_t$, resp. $\mathcal{U}^{X,1}_t,\mathcal{U}^{X,2}_t$, in (\ref{decompo-bipro}) belongs to the algebra $\ca_t$ generated by $(X_u)_{0\leq u\leq t}$, resp. to the tensor product $\ca_t\otimes \ca_t$, and conditions (\ref{regu-levy-area})-(\ref{chen}) only need to be satisfied for $\bu\in \ca_s\otimes \ca_s$. The latter restriction turns out to be fundamental when it comes to the construction of a product Lévy area above the free Brownian motion (see \cite[Section 5.1]{deya-schott}), the $q$-Brownian motion (see \cite[Section 3]{deya-schott-2}) or the non-commutative fractional Brownian motion (see \cite[Section 3]{deya-schott-3}). This is no longer the case in the finite-dimensional setting of the HfBm, thanks to the coordinates correspondence (\ref{correspondence-levy-areas}), and we thus got rid of the adaptedness conditions in the above formulation.
\end{remark}

\subsection{Canonical product Lévy area}

\

\smallskip

For every $d\geq 1$, we can apply the result of \cite[Theorem 2]{coutin-qian} to assert that the $(d(d+1))$-dimensional process 
$$\{x(i,j),\tilde{x}(i,j)\}_{1\leq j\leq i\leq d}$$
behind $X^{(d)}$ (along definition (\ref{defi-hfbm})) generates a canonical (classical) Lévy area. Denoting the components of this Lévy area as
\begin{equation}\label{1d-levy-areas}
\int_s^t \der x_{su}(i,j)\,  \mathrm{d}x_{u}(k,\ell) \ , \ \int_s^t \der x_{su}(i,j)\,  \mathrm{d}\tilde{x}_{u}(k,\ell) \ , \ \int_s^t  \der \tilde{x}_{su}(i,j)\,  \mathrm{d}x_{u}(k,\ell) \ , \ \int_s^t \der \tilde{x}_{su}(i,j) \,  \mathrm{d}\tilde{x}_{u}(k,\ell) \ ,
\end{equation}
we can then naturally lift this extension at the level of $X^{(d)}$ through the formula
\begin{equation}\label{levy-area-d}
\mathbf{X}^{\mathbf{2},(d)}_t((i,j),(k,\ell)):=\int_s^t \der X^{(d)}_{su}(i,j) \mathrm{d}X^{(d)}_u(k,\ell) \ , \quad 1\leq i,j,k,\ell \leq d \ ,
\end{equation}
which now admits a straightforward interpretation: for instance, if $i\geq j$ and $k\geq \ell$,
\small
\begin{align*}
&\mathbf{X}^{\mathbf{2},(d)}_{st}((i,j),(k,\ell))=\int_s^t \big( \der x_{su}(i,j)+\imath\,  \der \tilde{x}_{su}(i,j) \big)\,  \mathrm{d}\big( x_{u}(k,\ell)+\imath\,  \tilde{x}_{u}(k,\ell) \big):=\\
&\int_s^t \der x_{su}(i,j)\,  \mathrm{d}x_{u}(k,\ell)+\imath\int_s^t \der x_{su}(i,j)\,  \mathrm{d}\tilde{x}_{u}(k,\ell)+\imath \int_s^t  \der \tilde{x}_{su}(i,j)\,  \mathrm{d}x_{u}(k,\ell)-\int_s^t \der \tilde{x}_{su}(i,j) \,  \mathrm{d}\tilde{x}_{u}(k,\ell) \ .
\end{align*}
\normalsize
It is readily checked that $\mathbf{X}^{\mathbf{2},(d)}$ defines a (classical) Lévy area above $X^{(d)}$, with which we can immediately associate, through (\ref{correspondence-levy-areas}), a product Lévy area $\mathbf{X}^{(d)}$ above $X^{(d)}$. The resulting rough integral
\begin{equation}\label{int-canon-x-d}
\Big(\int_0^1 \bu_u \sharp \mathrm{d}\mathbb{X}^{(d)}_u\Big) \quad , \quad \text{where} \quad \mathbb{X}^{(d)}:=(X^d,\mathbf{X}^{(d)}) \ \text{and} \ \bu\in \mathbf{Q} \ ,
\end{equation}
then offers what can be regarded as a \enquote{canonical} interpretation of the integral against the HfBm of Hurst index $H>\frac13$.

\subsection{A matrix Itô--Stratonovich formula}

\

\smallskip

In the specific Brownian situation, that is when $H=\frac12$, the integrals in (\ref{1d-levy-areas}) can either be understood as Itô or as Stratonovich integrals. Let us respectively denote by $\mathbf{X}^{(d),\text{I}}$ and $\mathbf{X}^{(d),\text{S}}$ the product Lévy areas associated with each of these interpretations, and then set $\mathbb{X}^{(d),\text{I}}:=(X^{(d)},\mathbf{X}^{(d),\text{I}})$, $\mathbb{X}^{(d),\text{S}}:=(X^{(d)},\mathbf{X}^{(d),\text{S}})$. Owing to the consistency result of Lemma \ref{lem:consistency}, we can assert that for any \textit{adapted} controlled byprocess $\bu$ (i.e., $\bu$ and $\mathcal{U}^{X,1},\mathcal{U}^{X,2}$ in (\ref{decompo-bipro}) are adapted to the filtration generated by $\{x(i,j),\tilde{x}(i,j)\}_{1\leq j\leq i\leq d}$), the rough integral $\int_0^1 \bu_u \sharp \mathrm{d}\mathbb{X}^{(d),\text{I}}_u$, resp. $\int_0^1 \bu_u \sharp \mathrm{d}\mathbb{X}^{(d),\text{S}}_u$, coincides with the standard (componentwise) Itô, resp. Stratonovich, interpretation, a property which we can summarize as 
\begin{equation}\label{identif-pathwise}
\int_0^1 \bu_u \sharp \mathrm{d}\mathbb{X}^{(d),\text{I}}_u=\int_0^1 \bu_u \sharp \mathrm{d}X^{(d)}_u \quad , \quad \text{resp.}\quad \int_0^1 \bu_u \sharp \mathrm{d}\mathbb{X}^{(d),\text{S}}_u=\int_0^1 \bu_u \sharp (\circ  \mathrm{d}X^{(d)}_u) \ .
\end{equation}
Let us now rely on the above pathwise matrix approach (i.e., on the interpretation of these integrals as almost-sure limits of the sum in (\ref{corrected-riemmann-sums})) in order to establish an Itô--Stratonovich formula, that is a convenient description of the difference $\int_0^1 \bu_u \sharp (\circ  \mathrm{d}X^{(d)}_u) -\int_0^1 \bu_u \sharp \mathrm{d}X^{(d)}_u$. In fact, it is now clear that the fundamental difference between these two integrals lies at the level of the related product Lévy areas, and therefore we only need to focus on the difference $\mathbf{X}^{(d),\text{S}}_{st}-\mathbf{X}^{(d),\text{I}}_{st}$:

\begin{lemma}\label{lem:ito-strato-levy}
Assume that $H=\frac12$. Then, for all $0\leq s\leq t\leq 1$ and all random variables $U,V$ in $\ca^{(d)}$, one has almost surely
\begin{equation}\label{ito-strato-levy-areas}
\mathbf{X}^{(d),\text{S}}_{st}\big[ U\otimes V\big]=\mathbf{X}^{(d),\text{I}}_{st}\big[ U\otimes V\big]+\frac12 (t-s) \text{Tr}_d(V)\, U \ ,
\end{equation}
where $\text{Tr}_d(A):=\frac{1}{d} \sum_{i=1}^d A(i,i)$.
\end{lemma}

\begin{proof}
The identity essentially follows from the application of the classical 1d Itô--Stratonovich formula to the Lévy areas in (\ref{1d-levy-areas}), which, at the level of $X^{(d)}$, yields the two (almost sure) conversion formulas
\begin{equation}\label{conversion-1}
\int_s^t \der X^{(d)}_{su}(k,\ell_1)(\circ \mathrm{d}X^{(d)}_u(\ell_2,j))=\int_s^t \der X^{(d)}_{su}(k,\ell_1) \mathrm{d}X^{(d)}_u(\ell_2,j) \quad \text{if} \ (\ell_2,j)\neq (\ell_1,k) \ , 
\end{equation}
and
\begin{equation}\label{conversion-2}
\int_s^t \der X^{(d)}_{su}(k,\ell_1)(\circ \mathrm{d}X^{(d)}_u(\ell_1,k))=\int_s^t \der X^{(d)}_{su}(k,\ell_1)\mathrm{d}X^{(d)}_u(\ell_1,k)+\frac{1}{2d} (t-s) \ .
\end{equation}
Indeed, for (\ref{conversion-1}), observe first that if $\{\ell_2,j\}\neq \{\ell_1,k\}$, then the components of $X^{(d)}(k,\ell_1)$ and $X^{(d)}(\ell_2,j)$ are independent (complex) Brownian motions, so that the related Itô and Stratonovich integrals do coincide (a.s.). On the other hand, if $k>\ell_1$, one has a.s.
\begin{align*}
&\int_s^t \der X^{(d)}_{su}(k,\ell_1) (\circ \mathrm{d} X^{(d)}_u(k,\ell_1)) = \overline{\int_s^t \der X^{(d)}_{su}(\ell_1,k) (\circ \mathrm{d} X^{(d)}_u(\ell_1,k))}\\
&=\frac{1}{2d} \int_s^t (\der x_{su}(k,\ell_1)+\imath \, \der \tilde{x}_{su}(k,\ell_1)) \circ \mathrm{d}(x_u(k,\ell_1)+\imath \, \tilde{x}_u(k,\ell_1))\\
&=\frac{1}{2d} \bigg[ \Big\{\int_s^t \der x_{su}(k,\ell_1)\mathrm{d}x_u(k,\ell_1)+\frac12 (t-s)\Big\}+\imath \int_s^t \der x_{su}(k,\ell_1) \mathrm{d} \tilde{x}_u(k,\ell_1)\\
&\hspace{2cm}+\imath\int_s^t \der \tilde{x}_{su}(k,\ell_1)\mathrm{d}x_u(k,\ell_1)-\Big\{\int_s^t \der \tilde{x}_{su}(k,\ell_1) \mathrm{d}\tilde{x}_u(k,\ell_1)+\frac12(t-s) \Big\} \bigg]\\
&=\int_s^t \der X^{(d)}_{su}(k,\ell_1) \mathrm{d} X^{(d)}_u(k,\ell_1) \ .
\end{align*}
As for (\ref{conversion-2}), one has a.s., and along the same computations: for $k>\ell_1$,
\begin{align*}
&\int_s^t \der X^{(d)}_{su}(k,\ell_1) (\circ \mathrm{d} X^{(d)}_u(\ell_1,k)) = \overline{\int_s^t \der X^{(d)}_{su}(\ell_1,k) (\circ \mathrm{d} X^{(d)}_u(k,\ell_1))}\\
&=\frac{1}{2d} \int_s^t (\der x_{su}(k,\ell_1)+\imath \, \der \tilde{x}_{su}(k,\ell_1)) \circ \mathrm{d}(x_u(k,\ell_1)-\imath \, \tilde{x}_u(k,\ell_1))\\
&=\frac{1}{2d} \bigg[ \Big\{\int_s^t \der x_{su}(k,\ell_1)\mathrm{d}x_u(k,\ell_1)+\frac12 (t-s)\Big\}-\imath \int_s^t \der x_{su}(k,\ell_1) \mathrm{d} \tilde{x}_u(k,\ell_1)\\
&\hspace{2cm}+\imath\int_s^t \der \tilde{x}_{su}(k,\ell_1)\mathrm{d}x_u(k,\ell_1)+\Big\{\int_s^t \der \tilde{x}_{su}(k,\ell_1) \mathrm{d}\tilde{x}_u(k,\ell_1)+\frac12(t-s) \Big\} \bigg]\\
&=\int_s^t \der X^{(d)}_{su}(k,\ell_1) \mathrm{d} X^{(d)}_u(k,\ell_1)+\frac{1}{2d} (t-s) \ ,
\end{align*}
with a similar identity when $k=\ell_1$.

\smallskip

Based on (\ref{conversion-1})-(\ref{conversion-2}), we obtain a.s.
\begin{align*}
&\mathbf{X}^{(d),S}_{st}\big[U\otimes V\big](i,j)=\sum_{k,\ell_1,\ell_2=1}^d U(i,k)V(\ell_1,\ell_2) \int_s^t \der X^{(d)}_{su}(k,\ell_1) (\circ \mathrm{d}X^{(d)}_u(\ell_2,j))\\
&=\sum_{\substack{k,\ell_1,\ell_2=1\\(\ell_2,j)\neq (\ell_1,k)}}^d U(i,k)V(\ell_1,\ell_2) \int_s^t \der X^{(d)}_{su}(k,\ell_1) \mathrm{d}X^{(d)}_u(\ell_2,j)\\
& \hspace{1cm} +\sum_{\ell=1}^d U(i,j) V(\ell,\ell) \Big\{ \int_s^t \der X^{(d)}_{su}(k,\ell) \mathrm{d}X^{(d)}_u(\ell,k)+\frac{1}{2d} (t-s)\Big\}\\
&=\sum_{k,\ell_1,\ell_2=1}^d U(i,k)V(\ell_1,\ell_2) \int_s^t \der X^{(d)}_{su}(k,\ell_1)  \mathrm{d}X^{(d)}_u(\ell_2,j)+\frac12 (t-s)  \Big( \frac{1}{d} \sum_{\ell=1}^d V(\ell,\ell) \Big) U(i,j)\\
&=\mathbf{X}^{(d),I}_{st}\big[U\otimes V\big](i,j)+\frac12 (t-s) \text{Tr}_d(V) U(i,j) \ ,
\end{align*}
which corresponds to the desired identity.
\end{proof}

Injecting identity (\ref{ito-strato-levy-areas}) into (\ref{corrected-riemmann-sums}) immediately allows us to extend the conversion formula to a general level:
\begin{proposition}\label{prop:ito-strato}
Assume that $H=\frac12$. Then, for every adapted controlled process $\bu\in \mathbf{Q}$ with decomposition (\ref{decompo-bipro}), one has almost surely
\begin{equation}\label{ito-strato-matrix-gene}
\int_0^1 \bu_u \sharp (\circ  \mathrm{d}X^{(d)}_u)=\int_0^1 \bu_u \sharp \mathrm{d}X^{(d)}_u +\frac12 \int_0^1 \mathrm{d}u\, \big[ \id \times \text{Tr}_d \times \id\big] \big( \mathcal{U}^{X,1}_u+\mathcal{U}^{X,2}_u \big) \ .
\end{equation}
In particular, for all polynomials $P,Q$, one has almost surely
\begin{align}
&\int_0^1 P( X^{(d)}_u) (\circ \mathrm{d}X^{(d)}_u) Q( X^{(d)}_u)=\int_0^1 P( X^{(d)}_u) \mathrm{d}X^{(d)}_u Q( X^{(d)}_u)\nonumber\\
&\hspace{2cm} +\frac12 \int_0^1 \mathrm{d}u\, \big[ \id \times \text{Tr}_d \times \id\big] ( \partial P (X^{(d)}_u) \otimes Q(X^{(d)}_u)+P(X^{(d)}_u)\otimes \partial Q(X^{(d)}_u) ) \ .\label{ito-strato-matrix}
\end{align}
\end{proposition}

\

Identity (\ref{ito-strato-matrix-gene}), and even more explicitly identity (\ref{ito-strato-matrix}), thus corresponds to the matrix extension of the classical Itô--Stratonovich formula
\begin{align*}
\int_0^1 P( X^{(1)}_u) (\circ \mathrm{d}X^{(1)}_u) Q( X^{(1)}_u)=&\int_0^1 P( X^{(1)}_u) \mathrm{d}X^{(1)}_u Q( X^{(1)}_u)\\
&+\frac12 \int_0^1 \mathrm{d}u\, ( P'(X^{(1)}_u) Q(X^{(1)}_u)+P(X^{(1)}_u) Q'(X^{(1)}_u) ) \ .
\end{align*}
On the other hand, formulas (\ref{ito-strato-matrix-gene})-(\ref{ito-strato-matrix}) can somehow be seen as the finite-dimensional (and almost sure) counterpart of the Itô--Stratonovich formula for the free Brownian motion (see \cite[Proposition 5.6]{deya-schott}). The latter analogy will actually be emphasized through the convergence result in the subsequent Proposition \ref{prop:conv-integr-brown} (observe in particular the similarity between formulas (\ref{ito-strato-matrix}) and (\ref{strato-free})).

\begin{remark}
The above reasoning thus provides us with an illustration of the possibilities offered by the rough approach to stochastic integration with respect to $X^{(d)}$. Based on the interpretation of Proposition \ref{prop-int-gen}, we have indeed easily derived the general identity (\ref{ito-strato-matrix}) from the sole consideration of second-order objects. Proving this identity directly, that is through the stochastic componentwise interpretation of the integrals in (\ref{ito-strato-matrix}) (with use of the classical 1d Itô--Stratonovich conversion formula), would have been a much tougher task for polynomials $P,Q$ of high degrees.
\end{remark}

\section{From Hermitian to non-commutative fractional Brownian motion}\label{sec:asympt-results}

Our aim now is to study the transition, as the dimension parameter $d$ goes to infinity, from the HfBm to the so-called non-commutative fractional Brownian motion, and see how the convergence can be extended at the level of the related integrals. 

\smallskip

The asymptotic behaviour of the HfBm (or more precisely its \enquote{symmetric} counterpart) was already at the center of the investigations in \cite{matrix-approx}, at least when $H>\frac12$ and when focusing on the set of (random) measures $\{\mu^{(d)}_t\}_{t\geq 0}$ generated by the spectrum of $\{X^{(d)}_t\}_{t\geq 0}$ (note also that the convergence results of \cite{matrix-approx} do not apply to the integrals generated by $X^{(d)}$). 

\smallskip

We will here follow a slightly different approach and rather study convergence \emph{in the sense of non-commutative probability}, which, to our opinion, yields a better account of the mean spectral dynamics of the processes under consideration. Thus, as a first step, and for the sake of clarity, we need to briefly recall a few basics on the non-commutative probability setting (see \cite{nica-speicher} for more details).

\begin{definition}\label{defi:nc-probability-space}
We call a non-commutative probability space any pair $(\ca,\vp)$ such that:

\smallskip

\noindent
$(1)$ $\ca$ is a unital algebra over $\C$ endowed with an antilinear $\ast$-operation $X\mapsto X^\ast$ satisfying $(X^\ast)^\ast=X$ and $(XY)^\ast=Y^\ast X^\ast$ for all $X,Y\in \ca$. 

\smallskip

\noindent
$(2)$ $\vp:\ca \to \C$ is a \emph{positive trace} on $\ca$, that is a linear map satisfying $\vp(1)=1$, $\vp(XY)=\vp(YX)$ and $\vp(X^\ast X)\geq 0$ for all $X,Y\in \ca$.
\end{definition}

A classical way to \enquote{embed} the set of $d\times d$ random matrices (i.e., the set where $X^{(d)}$ lives) into such a structure is to consider the trace given by the mean value of the standard matrix trace. To be more specific, for each fixed $d\geq 1$, we focus on the unital algebra $\mathcal{M}_d(L^{\infty-}(\Omega))$ of matrices with complex random entries admitting finite moments of all orders, and set, for every $A\in \mathcal{M}_d(L^{\infty-}(\Omega))$,
\begin{equation}\label{defi:trace-vp-d}
\vp_d\big( A\big):=\frac{1}{d} \mathbb{E}\big[ \text{Tr}_d(A)\big] \ , \quad \text{where} \quad \text{Tr}_d(A):=\frac{1}{d} \sum_{i=1}^d A(i,i) \ .
\end{equation}
Beyond the fact that $\vp_d$ indeed satisfies the conditions in the above item $(2)$ (making the pair $(\mathcal{M}_d(L^{\infty-}(\Omega)),\vp_d)$ a non-commutative probability space), the interest in this particular trace lies of course in its close relation with the mean spectral distribution measure: for any $A\in \mathcal{M}_d(L^{\infty-}(\Omega))$ with (random) eigenvalues $\{\la_i(A)\}_{1\leq i\leq d}$, it is indeed readily checked that
$$\vp_d\big( A^r\big)=\mathbb{E}\Big[ \int_{\mathbb{C}} z^r \, \mu_A(\mathrm{d}z)\Big] \quad \text{where} \quad \mu_A:=\frac{1}{d} \sum_{i=1}^d \der_{\la_i(A)} \ .$$
Along this observation, a natural way to reach our objective, that is to catch (the asymptotic behaviour of) the mean spectral dynamics of the process $X^{(d)}$ is to study (the asymptotic behaviour of) the quantities 
\begin{equation}\label{joint-moments-x-d}
\vp_d\big( X^{(d)}_{t_1} \cdots X^{(d)}_{t_r}\big) \ ,
\end{equation}
for all possible $r\geq 1$ and $t_1,\ldots,t_r\geq 0$. For the same reasons, we will then be interested in the limit, as $d\to \infty$, of the \enquote{moments}
\begin{equation}\label{moments-integral-x-d}
\vp_d\Big( \Big( \int_0^1 P(X^{(d)}_u) dX^{(d)}_u Q(X^{(d)}_u) \Big)^r \Big) \quad , \quad r\geq 1 \ ,
\end{equation}
where the integral is defined through the considerations of the previous section.

\smallskip

Our description of the limits (for (\ref{joint-moments-x-d}) or for (\ref{moments-integral-x-d})) will again involve objects in some non-commutative probability space. To be more specific, we will rely on the following usual convergence interpretation:

\begin{definition}\label{defi:convergence}
Let $(\ca^{(n)},\vp^{(n)})$ be a sequence of non-commutative probability spaces, and let $U\in \ca$, where $(\ca,\vp)$ is another non-commutative probability space. A sequence of elements $U^{(n)} \in \ca^{(n)}$ is said to converge to $U$ (in the sense of non-commutative probability) if for every integer $r\geq 1$,
$$\vp^{(n)}\big( \big(U^{(n)}\big)^r\big) \stackrel{n\to \infty}{\longrightarrow} \vp\big( \big(U\big)^r\big) \ .$$
In a similar way, we say that a sequence of paths $\{Y^{(n)}_t\}_{t\geq 0}$ in $\ca^{(n)}$ converges (in the sense of non-commutative probability) to a path $\{Y_t\}_{t\geq 0}$ in $\ca$ if for every integer $r\geq 1$ and all times $t_1,\ldots,t_r \geq 0$,
$$\vp^{(n)}\big( Y^{(n)}_{t_1} \cdots Y^{(n)}_{t_r}\big) \stackrel{n\to\infty}{\longrightarrow} \vp\big( Y_{t_1} \cdots Y_{t_r}\big) \ .$$
\end{definition}

\

\subsection{An extension of Voiculescu's theorem}

\

\smallskip

Let us turn here to the presentation of the central combinatorial lemma that will serve us for the analysis of both (\ref{joint-moments-x-d}) and (\ref{moments-integral-x-d}). This result consists in fact in an easy extension of Voiculescu's fundamental theorem (\cite{voiculescu}) to a more general class of Gaussian matrices. For a clear statement, we need to introduce a few additional notations.


\smallskip

First, for every even integer $r\geq 1$, let us denote by $\mathcal{P}_2(r)$ the set of pairings of $\{1,\ldots,r\}$, i.e. the set of partitions of $\{1,\ldots ,r\}$ with blocks of two elements only, and by $NC_2(r)$ the subset of non-crossing pairings, i.e. the subset of pairings $\pi\in \mathcal{P}_2(r)$ for which there is no $1\leq p<q\leq r$ such that $\pi(p)> \pi(q)$.
Occasionally, we will identify a pairing $\pi \in \mathcal{P}_2(r)$ with a permutation of $\{1,\ldots,r\}$, by setting 
\begin{equation}\label{identif-pairing}
(\pi(p):=q,\pi(q):=p) \quad  \text{if and only if} \quad (p,q)\in \pi \ .
\end{equation} 

\smallskip

For every permutation $\si$ of $\{1,\ldots,r\}$, we denote by $\sharp(\si)$ the number of cycles in $\si$, and we recall that the genus of a pairing $\pi \in \mathcal{P}_2(r)$ (identified with a permutation along (\ref{identif-pairing})) is then defined by the formula
$$\text{genus}(\pi):=\frac12 \Big( \frac{r}{2}+1-\sharp(\ga \circ \pi) \Big) \ ,$$
where $\ga$ stands for the specific permutation of $\{1,\ldots,r\}$ given by $\ga:=(1 \, 2 \cdots r)$, i.e. $\ga(i)=i+1$ for $i=1,\ldots,r-1$ and $\ga(r)=1$. With this notation in hand, the three following properties, borrowed from \cite[Lecture 22]{nica-speicher}, turn out to be the keys toward Voiculescu's result:

\begin{lemma}\label{lem:voicu}
$(i)$ For all $r,d\geq 1$ and every permutation $\si$ of $\{1,\ldots,r\}$, it holds that
$$\sum_{i_1,\ldots,i_r=1}^d \1_{\{i_1=i_{\si(1)}\}}\cdots \1_{\{i_r=i_{\si(r)}\}} =d^{\sharp(\si)}\ .$$

\smallskip

\noindent
$(ii)$ For all even $r\geq 1$ and $\pi\in \mathcal{P}_2(r)$, one has $\text{genus}(\pi)\in \mathbb{N}$ and $0\leq \text{genus}(\pi)\leq \frac{r}{4}$. 

\smallskip

\noindent
$(iii)$ For all even $r\geq 1$ and $\pi\in \mathcal{P}_2(r)$, one has $\text{genus}(\pi)=0$ if and only if $\pi\in NC_2(r)$.
\end{lemma}


The desired central lemma now reads as follows (recall that, for the whole Section \ref{sec:asympt-results}, the notation $\vp_d$ refers to the trace on $\mathcal{M}_d(L^{\infty-}(\Omega))$ defined by (\ref{defi:trace-vp-d})):

\begin{lemma}\label{lem:gene}
Fix $d\geq 1$, as well as an arbitrary time-index set $I$, and consider, on a classical probability space $(\Omega,\mathcal{F},\mathbb{P})$, a Gaussian family $\{M_t(i,j)\}_{t\in I,1\leq i,j\leq d}$
of random variables with covariance of the form
\begin{equation}\label{cond-cova-matrix}
\mathbb{E}\big[ M_s(i,j)M_t(k,\ell) \big]=\frac{1}{d} c_M(s,t) \1_{i=\ell} \1_{j=k} \ ,
\end{equation}
for some time-covariance function $c_M:I^2 \to \R$. Then for all $r\geq 1$ and $t_1,\ldots,t_r\in I$, it holds that
\begin{equation}
\vp_d\big( M_{t_1} \cdots M_{t_r}\big)=\sum_{g=0}^{r/4} d^{-2g}\sum_{\substack{\pi \in \mathcal{P}_2(r)\\ \text{genus}(\pi)=g}} \prod_{(p,q)\in \pi} c_M(t_p,t_q)  \ .
\end{equation}
\end{lemma}
\begin{proof}
As we evoked it earlier, the argument is a mere adaptation of the proof of Voiculescu's fundamental result. To be more specific, we will follow the lines of the proof of \cite[Theorem 22.24]{nica-speicher}. Let us first expand the quantity under consideration using the standard Wick formula for Gaussian variables, which, combined with (\ref{cond-cova-matrix}), yields 
\begin{align*}
\vp_d\big( M_{t_1} \cdots  M_{t_r}\big)&=\frac{1}{d} \sum_{i_1,\ldots,i_r=1}^d \mathbb{E} \big[ M_{t_1}(i_1,i_2) M_{t_2}(i_2,i_3) \cdots M_{t_r}(i_r,i_1) \big]\\
&= \frac{1}{d} \sum_{i_1,\ldots,i_r=1}^d \sum_{\pi \in \mathcal{P}_2(r)} \prod_{(p,q)\in \pi} \mathbb{E} \big[ M_{t_p}(i_p,i_{p+1}) M_{t_q}(i_q,i_{q+1}) \big]\\
&=\sum_{\pi \in \mathcal{P}_2(r)}  \Big( \frac{1}{d^{1+\frac{r}{2}}} \sum_{i_1,\ldots,i_r=1}^d \prod_{(p,q)\in \pi} \1_{\{i_p=i_{q+1}\}} \1_{\{i_{p+1}=i_q\}}\Big)\Big(\prod_{(p,q)\in \pi} c_H(t_p,t_q) \Big) \ ,
\end{align*}
where we have used the convention $i_{r+1}:=i_1$. Identifying pairings with permutations along (\ref{identif-pairing}) and using the specific permutation $\ga:=(1\, 2 \ldots r)$, we can easily rewrite the previous quantity as 
\begin{align*}
\vp_d\big( M_{t_1} \cdots  M_{t_r}\big)&=\sum_{\pi \in \mathcal{P}_2(r)} \Big( \frac{1}{d^{1+\frac{r}{2}}} \sum_{i_1,\ldots,i_r=1}^d \prod_{(p,q)\in \pi} \1_{\{i_p=i_{(\ga \circ \pi)(p)}\}} \1_{\{i_{(\ga \circ \pi)(q)}=i_q\}} \Big) \Big(\prod_{(p,q)\in \pi} c_H(t_p,t_q) \Big) \\
&=\sum_{\pi \in \mathcal{P}_2(r)} \Big( \frac{1}{d^{1+\frac{r}{2}}} \sum_{i_1,\ldots,i_r=1}^d \1_{\{i_1=i_{(\ga\circ \pi)(1)}\}}\cdots \1_{\{i_r=i_{(\ga\circ \pi)(r)}\}} \Big) \Big(\prod_{(p,q)\in \pi} c_H(t_p,t_q) \Big) \ .
\end{align*}
Finally, we can successively lean on the results of items $(i)$ and $(ii)$ in Lemma \ref{lem:voicu} to assert that
\begin{align*}
\vp_d\big( M_{t_1} \cdots  M_{t_r}\big)
&= \sum_{\pi \in \mathcal{P}_2(r)} d^{-2\, \text{genus}(\pi)} \Big(\prod_{(p,q)\in \pi} c_H(t_p,t_q) \Big)\\
&=\sum_{g=0}^{r/4} d^{-2g}\sum_{\substack{\pi \in \mathcal{P}_2(r)\\ \text{genus}(\pi)=g}}  \Big(\prod_{(p,q)\in \pi} c_H(t_p,t_q) \Big) \ .
\end{align*}
\end{proof}

The covariance of the Gaussian family $\{X^{(d)}_t(i,j)\}_{1\leq i,j\leq d,t\geq 0}$ generated by the HfBm is precisely of the form (\ref{cond-cova-matrix}). To be more specific, it is readily checked that for all $s,t\geq 0$ and $1\leq i,j,k,\ell\leq d$, one has
\begin{equation}
\mathbb{E}\big[ X^{(d)}_s(i,j)X^{(d)}_t(k,\ell) \big]=\frac{1}{d} c_H(s,t) \1_{i=\ell} \1_{j=k} \ ,
\end{equation}
where $c_H$ refers to the classical fractional covariance of index $H$, that is 
\begin{equation}\label{cova-NC-fBm}
c_H(s,t):=\frac12 \big\{s^{2H}+t^{2H}-|t-s|^{2H}\big\} \ .
\end{equation}
We are thus in a position to apply Lemma \ref{lem:gene} and assert that for all $r\geq 1$, $t_1,\ldots,t_r\geq 0$, 
\begin{equation}\label{moments-x-d-descrip}
\vp_d\big( X^{(d)}_{t_1} \cdots X^{(d)}_{t_r}\big)=\sum_{g=0}^{r/4} d^{-2g}\sum_{\substack{\pi \in \mathcal{P}_2(r)\\ \text{genus}(\pi)=g}} \prod_{(p,q)\in \pi} c_H(t_p,t_q)  \ .
\end{equation}
The identification of the limit of $X^{(d)}$ (as $d$ goes to infinity) is now straightforward. Appealing indeed to the result of item $(iii)$ in Lemma \ref{lem:voicu}, we are naturally led to the consideration of the \textit{non-commutative fractional Brownian motion}:

\begin{definition}\label{defi:nc-bm}
In a NC-probability space $(\ca^{(\infty)},\vp_\infty)$, and for every $H\in (0,1)$, we call a non-commutative fractional Brownian motion (NC-fBm) of Hurst index $H$ any collection $\{X^{(\infty)}_t\}_{t\geq 0}$ of self-adjoint elements in $\ca^{(\infty)}$ such that, for every even integer $r\geq $1 and all $t_1,\ldots,t_r \geq 0$, one has
\begin{equation}\label{form-wick}
\vp_\infty\big( X^{(\infty)}_{t_1}\cdots X^{(\infty)}_{t_r}\big)=\sum_{\pi \in NC_2(r)} \prod_{\{p,q\}\in \pi}c_H(t_p,t_q)\ ,
\end{equation}
and $\vp_\infty\big( X^{(\infty)}_{t_1}\cdots X^{(\infty)}_{t_r}\big)=0$ whenever $r$ is an odd integer.
\end{definition}

For every fixed $H\in (0,1)$, the existence of such a NC-fBm (living in some non-commutative probability space) is guaranteed by the general results of \cite{q-gauss}. Letting $d$ tend to infinity in (\ref{moments-x-d-descrip}), we immediately get, thanks to Lemma \ref{lem:voicu}, item $(iii)$:
\begin{proposition}\label{prop:convergence-result}
For every $H\in (0,1)$, and as $d\to \infty$,  $X^{(d)}$ converges, in the sense of non-commutative probability, to a NC-fBm of same Hurst index $H$. 
\end{proposition}

\

Our objective in the sequel is to show that, at least when $H\geq \frac12$, this convergence result can be extended at the level of the integral driven by $X^{(d)}$, as defined in (\ref{int-canon-x-d}). The limit will naturally involve some integral driven by a NC-fBm $X^{(\infty)}$, that we will be able to interpret thanks to the results of \cite{deya-schott-3}.

\subsection{Convergence of the integral in the Young case: $H>\frac12$}

\

\smallskip

The aim here is to establish the following (expected) result:

\begin{proposition}\label{prop:conv-integr}
When $H>\frac12$, and for all polynomials $P,Q$, one has, in the sense of non-commutative probability,
\begin{equation}\label{conv-young}
\int_0^1 P\big( X^{(d)}_u\big) \mathrm{d}X^{(d)}_u Q\big( X^{(d)}_u\big)\stackrel{d\to \infty}{\longrightarrow} \int_0^1 P\big( X^{(\infty)}_u\big) \mathrm{d}X^{(\infty)}_u Q\big( X^{(\infty)}_u\big) \ ,
\end{equation}
where the integral in the left-hand side is interpreted via (\ref{int-canon-x-d}) and the integral in the limit is interpreted through \cite[Proposition 2.5]{deya-schott-3}.
\end{proposition}

Our strategy toward (\ref{conv-young}) consists in trying to reduce the problem to the polynomial convergence of Proposition \ref{prop:convergence-result} (that is, the convergence of all joint finite \enquote{moments} of $X^{(d)}$, in the sense of Definition \ref{defi:convergence}). We will thus rely on a polynomial approximation of the integrals in (\ref{conv-young}):

\

\begin{lemma}
Assume that $H> \frac12$. Then for all fixed $1\leq d\leq \infty$, $1\leq r <\infty$ and all polynomials $P,Q$, it holds that
\begin{equation}\label{fixed-d-1}
\vp_d\Big( \Big( \int_0^1 P\big( X^{(d)}_u\big) \mathrm{d}X^{(d)}_u Q\big( X^{(d)}_u\big)\Big)^r\Big)=\lim_{n\to \infty} \vp_d\Big( \Big( \sum_{i=0}^{2^n-1}P\big( X^{(d)}_{t_i^n}\big) \der X^{(d)}_{t_i^n t_{i+1}^n} Q\big( X^{(d)}_{t_i^n}\big)\Big)^r\Big) \ ,
\end{equation}
where $t_i^n:=\frac{i}{2^n}$ for $i=0,\ldots,2^n$.
\end{lemma}

\begin{proof}
For $1\leq d< \infty$, observe first that for all $A,B\in \mathcal{M}_d(L^{\infty-}(\Omega))$,
\begin{equation}\label{link-classical-norm}
\big|\vp_d\big(A^r-B^r\big) \big|\leq \frac{1}{d}\mathbb{E}\big[ \|A-B\|^2\big]^{\frac12} \sum_{m=0}^{r-1} \mathbb{E}\big[ \|A\|^{4m}\big]^{\frac14}\mathbb{E}\big[ \|B\|^{4(r-1-m)}\big]^{\frac14} \ ,
\end{equation}
which, by Proposition \ref{prop-int-gen} and Remark \ref{rk:convergence-l-p}, entails that
\small
\begin{align}
&\vp_d\Big( \Big( \int_0^1 P\big( X^{(d)}_u\big) \mathrm{d}X^{(d)}_u Q\big( X^{(d)}_u\big)\Big)^r\Big)\nonumber\\
&=\lim_{n\to \infty} \vp_d\Big( \Big( \sum_{i=0}^{2^n-1}\Big\{  P\big( X^{(d)}_{t_i^n}\big) \der X^{(d)}_{t_i^n t_{i+1}^n} Q\big( X^{(d)}_{t_i^n}\big)+[\mathbf{X}^{(d)}_{t^n_it^n_{i+1}}\times \id](\mathcal{U}^{X,1}_{t^n_i})+[\id \times \mathbf{X}^{(d),\ast}_{t^n_it^n_{i+1}}](\mathcal{U}^{X,2}_{t^n_i})\Big\}\Big)^r\Big)\ ,\label{sum-correc}
\end{align}
\normalsize
with $\mathcal{U}^{X,1}_s := \partial P(X^{(d)}_s)  \otimes Q(X^{(d)}_s)$ and $\mathcal{U}^{X,2}_s:=P( X^{(d)}_s) \otimes \partial Q(X^{(d)}_s)$. Now we can lean on standard estimates for the canonical Lévy areas (\ref{1d-levy-areas}) behind $\mathbf{X}^{(d)}$ to assert that for all $0< \ka <H$ and $0\leq s\leq t\leq 1$, 
$$\big| [\mathbf{X}^{(d)}_{st}\times \id](\mathcal{U}^{X,1}_{s}) \big| \leq C_{d,\ka} |t-s|^{2\ka} \quad \text{and} \quad \big|[\id \times \mathbf{X}^{(d),\ast}_{st}](\mathcal{U}^{X,2}_{s})\big| \leq C_{d,\ka} |t-s|^{2\ka} \ , $$
for some random variable $C_{d,\ka}$ admitting finite moments of any order. Picking $\ka\in (\frac12,H)$ and using again (\ref{link-classical-norm}), we can conclude that the limit of the sum in (\ref{sum-correc}) reduces in fact to the limit of the Riemann sum in (\ref{fixed-d-1}), as desired.

\smallskip

For $d=\infty$, the convergence is a straightforward consequence of \cite[Proposition 2.5]{deya-schott-3}.
\end{proof}

\begin{proof}[Proof of Proposition \ref{prop:conv-integr}]
Based on Proposition \ref{prop:convergence-result} and approximation (\ref{fixed-d-1}) (which holds for both finite $d$ and for $d=\infty$), the problem reduces to justifying some limit interchange, and to this end, we are going to show that
\small 
\begin{equation}\label{unifo-conv}
\sup_{n\geq 0}\, \bigg| \vp_d\Big( \Big( \sum_{i=0}^{2^n-1}P\big( X^{(d)}_{t_i^n}\big) \der X^{(d)}_{t_i^n t_{i+1}^n} Q\big( X^{(d)}_{t_i^n}\big)\Big)^r\Big)- \vp_\infty\Big( \Big( \sum_{i=0}^{2^n-1}P\big( X_{t_i^n}\big) \der X^{(\infty)}_{t_i^n t_{i+1}^n} Q\big( X_{t_i^n}\big)\Big)^r\Big)\bigg| \stackrel{d\to \infty}{\longrightarrow} 0 \ .
\end{equation}
\normalsize
By Proposition \ref{prop:convergence-result}, this convergence is known to be true for every fixed $n\geq 0$. For a uniform result, let us first write, if $P(X)=\sum_{p\geq 0} a_p X^p$ and $Q(X)=\sum_{q\geq 0} b_q X^q$, 
\begin{align}
&\vp_d\Big( \Big( \sum_{i=0}^{2^n-1}P\big( X^{(d)}_{t_i^n}\big) \der X^{(d)}_{t_i^n t_{i+1}^n} Q\big( X^{(d)}_{t_i^n}\big)\Big)^r\Big)=\sum_{i_1,\ldots,i_r=0}^{2^n-1} \sum_{p_1,\ldots,p_r \geq 0} \sum_{q_1,\ldots,q_r\geq 0} a_{p_1} \cdots a_{p_r} b_{q_1} \cdots b_{q_r}\nonumber\\
&\hspace{4cm} \vp_d \Big( \big\{ \big( X^{(d)}_{t_{i_1}}\big)^{p_1} \der X^{(d)}_{t_{i_1}t_{i_1+1}}\big( X^{(d)}_{t_{i_1}}\big)^{q_1} \big\} \cdots \big\{ \big( X^{(d)}_{t_{i_r}}\big)^{p_r} \der X^{(d)}_{t_{i_r}t_{i_r+1}}\big( X^{(d)}_{t_{i_r}}\big)^{q_r} \big\} \Big) \ .\label{first-expansion}
\end{align}
Then observe that the covariance of the Gaussian family
\begin{equation}\label{gauss-fam-proof}
\big\{ X^{(d)}_{t_{i_1}}(k,\ell), \der X^{(d)}_{t_{i_1}t_{i_1+1}}(k,\ell),\ldots,X^{(d)}_{t_{i_r}}(k,\ell), \der X^{(d)}_{t_{i_r}t_{i_r+1}}(k,\ell)\big\}_{1\leq k,\ell\leq d}
\end{equation}
involved in (\ref{first-expansion}) can clearly be written as in (\ref{cond-cova-matrix}), for some suitable time-covariance function $c$, and therefore we can apply Lemma \ref{lem:gene} to deduce that, if $R:=(p_1+q_1+1)+\ldots+(p_r+q_r+1)$,
\begin{equation}\label{lhs}
\vp_d \Big( \big\{ \big( X^{(d)}_{t_{i_1}}\big)^{p_1} \der X^{(d)}_{t_{i_1}t_{i_1+1}}\big( X^{(d)}_{t_{i_1}}\big)^{q_1} \big\} \cdots \big\{ \big( X^{(d)}_{t_{i_r}}\big)^{p_r} \der X^{(d)}_{t_{i_r}t_{i_r+1}}\big( X^{(d)}_{t_{i_r}}\big)^{q_r} \big\} \Big) =\sum_{g=0}^{ R/4} d^{-2g} \sum_{\substack{\pi \in \mathcal{P}_2(R)\\\text{genus}(\pi)=g}} C_\pi \ ,
\end{equation}
where $C_\pi=C_\pi((p_1,\ldots,p_r),(q_1,\ldots,q_r),(t_{i_1},\ldots,t_{i_r}))$ is naturally obtained as the product (along the pairs in $\pi$) of the time-covariances associated with the product of variables in the left-hand side.

\smallskip

Noting that the time-covariance function $c$ of the family in (\ref{gauss-fam-proof}) is in fact nothing but the covariance of the family
\begin{equation*}
\big\{ X^{(\infty)}_{t_{i_1}}, \der X^{(\infty)}_{t_{i_1}t_{i_1+1}},\ldots,X^{(\infty)}_{t_{i_r}}, \der X^{(\infty)}_{t_{i_r}t_{i_r+1}}\big\} \ ,
\end{equation*}
we can write
\small
\begin{align*}
&\vp_d \Big( \big\{ \big( X^{(d)}_{t_{i_1}}\big)^{p_1} \der X^{(d)}_{t_{i_1}t_{i_1+1}}\big( X^{(d)}_{t_{i_1}}\big)^{q_1} \big\} \cdots \big\{ \big( X^{(d)}_{t_{i_r}}\big)^{p_r} \der X^{(d)}_{t_{i_r}t_{i_r+1}}\big( X^{(d)}_{t_{i_r}}\big)^{q_r} \big\} \Big)\\
&-\vp_\infty \Big( \big\{ \big( X^{(\infty)}_{t_{i_1}}\big)^{p_1} \der X^{(\infty)}_{t_{i_1}t_{i_1+1}}\big( X^{(\infty)}_{t_{i_1}}\big)^{q_1} \big\} \cdots \big\{ \big( X^{(\infty)}_{t_{i_r}}\big)^{p_r} \der X^{(\infty)}_{t_{i_r}t_{i_r+1}}\big( X^{(\infty)}_{t_{i_r}}\big)^{q_r} \big\} \Big)= \sum_{g=1}^{R/4} d^{-2g} \sum_{\substack{\pi \in \mathcal{P}_2(R)\\\text{genus}(\pi)=g}} C_\pi \ .
\end{align*}
\normalsize
At this point, let us recall that since $H> \frac12$, the increments of a fractional Brownian motion of Hurst index $H$ are positively correlated, that is, if $x$ is such a fractional process (defined on some classical probability space $(\Omega,\mathcal{F},\mathbb{P})$), one has $\mathbb{E}\big[ \der x_{u_1 u_2}\der x_{v_1 v_2}\big]\geq 0$ for all $u_1\leq u_2$, $v_1\leq v_2$. In particular,  the time-covariances involved in $C_\pi$ are all positive, making $C_\pi$ positive too, so that for all $g\geq 1$,
\begin{equation}\label{bound-c-pi}
\Big| \sum_{\substack{\pi \in \mathcal{P}_2(R)\\\text{genus}(\pi)=g}} C_\pi \Big| \leq \sum_{\pi \in \mathcal{P}_2(R)} C_\pi \ .
\end{equation}
Going back to (\ref{first-expansion}), we thus have, setting $\text{d}_P:=\text{deg}(P)$ and $\text{d}_Q:=\text{deg}(Q)$,
\begin{align*}
&\bigg| \vp_d\Big( \Big( \sum_{i=0}^{2^n-1}P\big( X^{(d)}_{t_i^n}\big) \der X^{(d)}_{t_i^n t_{i+1}^n} Q\big( X^{(d)}_{t_i^n}\big)\Big)^r\Big)- \vp_\infty\Big( \Big( \sum_{i=0}^{2^n-1}P\big( X^{(\infty)}_{t_i^n}\big) \der X^{(\infty)}_{t_i^n t_{i+1}^n} Q\big( X^{(\infty)}_{t_i^n}\big)\Big)^r\Big)\bigg|\\
&\leq \Big( \sum_{g=1}^{\frac{r}{4}(\text{d}_P+\text{d}_Q)} d^{-2g} \Big) \Big( \sum_{i_1,\ldots,i_r=0}^{2^n-1} \sum_{p_1,\ldots,p_r \geq 0} \sum_{q_1,\ldots,q_r\geq 0} |a_{p_1}| \cdots |a_{p_r}| |b_{q_1}| \cdots |b_{q_r}| \sum_{\pi \in \mathcal{P}_2(R)} C_\pi\Big)\\
&\leq \Big( \sum_{g=1}^{\frac{r}{4}(\text{d}_P+\text{d}_Q)} d^{-2g} \Big)  \mathbb{E} \bigg[ \Big( \sum_{i=0}^{2^n-1} |P|(x_{t_i}) \der x_{t_i t_{i+1}} |Q|(x_{t_i})\Big)^r \bigg] \ ,
\end{align*}
where the polynomials $|P|,|Q|$ are defined as $|P|(X)=\sum_{p\geq 0} |a_p| X^p$ and $|Q|(X)=\sum_{q\geq 0} |b_q| X^q$. The uniform convergence statement (\ref{unifo-conv}) now comes from the fact that, by standard results on Young integration with respect to the fractional Brownian motion of index $H>\frac12$, one has
\begin{equation}\label{reduc-1-d-1}
\sup_{n\geq 0}\, \mathbb{E} \bigg[ \Big( \sum_{i=0}^{2^n-1} |P|(x_{t_i}) \der x_{t_i t_{i+1}} |Q|(x_{t_i})\Big)^r \bigg] \ < \ \infty\ .
\end{equation}
Once endowed with (\ref{unifo-conv}), we can successively assert that
\smallskip
\begin{align}
&\lim_{d\to \infty}\vp_d\Big( \Big( \int_0^1 P\big( X^{(d)}_u\big) \mathrm{d}X^{(d)}_u Q\big( X^{(d)}_u\big)\Big)^r\Big)\label{limits-switch-start}\\
&=\lim_{d\to \infty} \lim_{n\to \infty} \vp_d\Big( \Big( \sum_{i=0}^{2^n-1}P\big( X^{(d)}_{t_i^n}\big) \der X^{(d)}_{t_i^n t_{i+1}^n} Q\big( X^{(d)}_{t_i^n}\big)\Big)^r\Big) \quad \quad \text{(by (\ref{fixed-d-1}) with $d<\infty$)}\nonumber\\
&=\lim_{n\to \infty} \lim_{d\to \infty} \vp_d\Big( \Big( \sum_{i=0}^{2^n-1}P\big( X^{(d)}_{t_i^n}\big) \der X^{(d)}_{t_i^n t_{i+1}^n} Q\big( X^{(d)}_{t_i^n}\big)\Big)^r\Big) \quad \quad \text{(by (\ref{unifo-conv}))} \nonumber\\
&=\lim_{n\to \infty} \vp_\infty\Big( \Big( \sum_{i=0}^{2^n-1}P\big( X^{(\infty)}_{t_i^n}\big) \der X^{(\infty)}_{t_i^n t_{i+1}^n} Q\big( X^{(\infty)}_{t_i^n}\big)\Big)^r\Big)\quad \quad \text{(by Proposition \ref{prop:convergence-result})} \nonumber\\
&=\vp_\infty\Big( \Big( \int_0^1 P\big( X^{(\infty)}_u\big) \mathrm{d}X^{(\infty)}_u Q\big( X^{(\infty)}_u\big)\Big)^r\Big)\quad \quad \text{(by (\ref{fixed-d-1}) with $d=\infty$)}\label{limits-switch-end} \ ,
\end{align}
\normalsize
which corresponds to the desired conclusion.
\end{proof}

\subsection{Convergence of the integral in the Brownian case}

\

\smallskip

Recall that when $H=\frac12$ and for $1\leq d<\infty$, the integral with respect to $X^{(d)}$ can either be interpreted in the Itô sense (and denoted by $\int_0^1 \bu_u \sharp \mathrm{d}X^{(d)}_u$) or in the Stratonovich sense (and denoted by $\int_0^1 \bu_u \sharp (\circ  \mathrm{d}X^{(d)}_u)$). The main convergence result in this situation can then be stated as follows:

\begin{proposition}\label{prop:conv-integr-brown}
When $H=\frac12$, and for all polynomials $P,Q$, one has, in the sense of non-commutative probability,
\begin{equation}\label{conv-bro-ito}
\int_0^1 P\big( X^{(d)}_u\big) \mathrm{d}X^{(d)}_u Q\big( X^{(d)}_u\big)\stackrel{d\to \infty}{\longrightarrow} \int_0^1 P\big( X^{(\infty)}_u\big) \mathrm{d}X^{(\infty)}_u Q\big( X^{(\infty)}_u\big) \ ,
\end{equation}
where the latter integral is interpreted through the result of \cite[Proposition 2.6]{deya-schott-3} (or equivalently, through the result of \cite[Corollary 3.2.2]{biane-speicher}). Besides,
\begin{equation}\label{conv-bro-strato}
\int_0^1 P\big( X^{(d)}_u\big) (\circ \mathrm{d}X^{(d)}_u) Q\big( X^{(d)}_u\big)\stackrel{d\to \infty}{\longrightarrow}  \int_0^1 P\big( X^{(\infty)}_u\big) (\circ \mathrm{d}X^{(\infty)}_u) Q\big( X^{(\infty)}_u\big) \ ,
\end{equation}
where the latter \enquote{Stratonovich} integral is defined as
\begin{align}
&\int_s^t P(X^{(\infty)}_u) (\circ \mathrm{d}X^{(\infty)}_u) Q(X^{(\infty)}_u):=\int_s^t P(X^{(\infty)}_u) \mathrm{d}X^{(\infty)}_u Q(X^{(\infty)}_u)\nonumber\\
&\hspace{1cm}+\frac12 \int_s^t \mathrm{d}u\big( \id \times \vp_\infty \times \id\big)\big[ \partial P (X^{(\infty)}_u) \otimes Q(X^{(\infty)}_u)+P(X^{(\infty)}_u)\otimes \partial Q(X^{(\infty)}_u) \big] \ .\label{strato-free}
\end{align}
\end{proposition}

\

Let us again point out the similarity between the two Itô--Stratonovich formulas (\ref{ito-strato-matrix}) and (\ref{strato-free}), and the underlying \enquote{transformation} of the matrix trace $\text{Tr}_d$ in (\ref{ito-strato-matrix}) into the non-commutative trace $\vp_\infty$ in (\ref{strato-free}).

\smallskip

In order to show (\ref{conv-bro-ito}) and (\ref{conv-bro-strato}), we will use the same general idea as in the previous section, namely a polynomial approximation of the integrals. Our starting result is here the following one:

\begin{lemma}
Assume that $H= \frac12$. Then for all fixed $1\leq d\leq \infty$, $1\leq r <\infty$ and all polynomials $P,Q$, it holds that
\begin{equation}\label{fixed-d-1-brown}
\vp_d\Big( \Big( \int_0^1 P\big( X^{(d)}_u\big) \mathrm{d}X^{(d)}_u Q\big( X^{(d)}_u\big)\Big)^r\Big)=\lim_{n\to \infty} \vp_d\Big( \Big( \sum_{i=0}^{2^n-1}P\big( X^{(d)}_{t_i^n}\big) \der X^{(d)}_{t_i^n t_{i+1}^n} Q\big( X^{(d)}_{t_i^n}\big)\Big)^r\Big) 
\end{equation}
and 
\begin{equation}\label{fixed-d-1-brown-strato}
\vp_d\Big( \Big( \int_0^1 P\big( X^{(d)}_u\big) (\circ \mathrm{d}X^{(d)}_u) Q\big( X^{(d)}_u\big)\Big)^r\Big)=\lim_{n\to \infty} \vp_d\Big( \Big(  \int_0^1 P\big( X^{(d,n)}_u\big) \mathrm{d} X^{(d,n)}_u Q\big( X^{(d,n)}_u\big)\Big)^r\Big) \ ,
\end{equation}
where, for each $1\leq d\leq \infty$ and $n\geq 1$, $X^{(d,n)}$ stands for the linear interpolation of $X^{(d)}$ along the subdivision $(t_i^n)$, that is
\begin{equation}\label{linear-interpolation}
X^{(d,n)}_t:=X^{(d)}_{t_i^n}+2^n(u-t_i^n) \der X^{(d)}_{t_i^n t_{i+1}^n} \quad \text{for} \ \ t\in [t_i^n,t_{i+1}^n] \ .
\end{equation}
\end{lemma}

\begin{proof}
For $1\leq d< \infty$, we can use (\ref{link-classical-norm}) to reduce the problem to the consideration of the moments
$$\mathbb{E}\Big[ \Big\| \int_0^1 P\big( X^{(d)}_u\big) \mathrm{d}X^{(d)}_u Q\big( X^{(d)}_u\big)-\sum_{i=0}^{2^n-1}P\big( X^{(d)}_{t_i^n}\big) \der X^{(d)}_{t_i^n t_{i+1}^n} Q\big( X^{(d)}_{t_i^n}\big)\Big\|^r \Big]$$
and
$$\mathbb{E}\Big[ \Big\| \int_0^1 P\big( X^{(d)}_u\big) (\circ \mathrm{d}X^{(d)}_u) Q\big( X^{(d)}_u\big)-\int_0^1 P\big( X^{(d,n)}_u\big) \mathrm{d} X^{(d,n)}_u Q\big( X^{(d,n)}_u\big)\Big\|^r \Big]$$
for any $r\geq 1$. The fact that these quantities converge to $0$ as $n\to\infty$ (for any $r\geq 1$) is then a standard approximation result from Brownian analysis, and (\ref{fixed-d-1-brown})-(\ref{fixed-d-1-brown-strato}) immediately follow.

\smallskip

For $d=\infty$, the assertion can be readily derived from \cite[Proposition 2.6]{deya-schott-3}.
\end{proof}

\

\begin{proof}[Proof of Proposition \ref{prop:conv-integr-brown}]
Endowed with (\ref{fixed-d-1-brown})-(\ref{fixed-d-1-brown-strato}), and keeping the polynomial convergence of Proposition \ref{prop:convergence-result} in mind, it suffices, as before, to show that
\small
\begin{equation}\label{unifo-conv-ito}
\sup_{n\geq 0}\, \bigg| \vp_d\Big( \Big( \sum_{i=0}^{2^n-1}P\big( X^{(d)}_{t_i^n}\big) \der X^{(d)}_{t_i^n t_{i+1}^n} Q\big( X^{(d)}_{t_i^n}\big)\Big)^r\Big)- \vp_\infty\Big( \Big( \sum_{i=0}^{2^n-1}P\big( X_{t_i^n}\big) \der X^{(\infty)}_{t_i^n t_{i+1}^n} Q\big( X_{t_i^n}\big)\Big)^r\Big)\bigg| \stackrel{d\to \infty}{\longrightarrow} 0 \ .
\end{equation}
\normalsize
and that
\small
\begin{align}
\sup_{n\geq 0}\, \bigg|& \vp_d\Big( \Big( \int_0^1 P\big( X^{(d,n)}_u\big) \mathrm{d} X^{(d,n)}_u Q\big( X^{(d,n)}_u\big)\Big)^r\Big)\nonumber\\
&\hspace{1cm}- \vp_\infty\Big( \Big( \int_0^1 P\big( X^{(\infty,n)}_u\big) (\circ \mathrm{d}X^{(\infty,n)}_u) Q\big( X^{(\infty,n)}_u\big)\Big)^r\Big)\bigg| \stackrel{d\to \infty}{\longrightarrow} 0 \ .\label{unifo-conv-strato}
\end{align}
\normalsize
For (\ref{unifo-conv-ito}), we can easily follow the lines of the proof of Proposition \ref{prop:conv-integr}, using the fact that the increments of a (standard) Brownian motion $x$ are also positively correlated, together with the standard uniform control
$$\sup_{n\geq 0}\, \mathbb{E} \bigg[ \Big( \sum_{i=0}^{2^n-1} |P|(x_{t_i}) \der x_{t_i t_{i+1}} |Q|(x_{t_i})\Big)^r \bigg] \ < \ \infty\ ,$$
for any $r\geq 1$.

\smallskip

In fact, this strategy can be applied to prove (\ref{unifo-conv-strato}) as well, by noting on the one hand that the expansion of the moment
\small
\begin{align*}
&\vp_d\Big( \Big( \int_0^1 P\big( X^{(d,n)}_u\big) \mathrm{d} X^{(d,n)}_u Q\big( X^{(d,n)}_u\big)\Big)^r\Big)\\
&=\vp_d\Big( \Big( \sum_{i=0}^{2^n-1}\int_{t_i^n}^{t_{i+1}^n}du\,  P\big( X^{(d)}_{t_i^n}+2^n(u-t_i^n) \der X^{(d)}_{t_i^n t_{i+1}^n}\big) \big( 2^n \der X^{(d)}_{t_i^nt_{i+1}^n}\big) Q\big( X^{(d)}_{t_i^n}+2^n(u-t_i^n) \der X^{(d)}_{t_i^n t_{i+1}^n}\big)\Big)^r\Big)
\end{align*}
\normalsize
still gives rise to the consideration of positively-correlated variables, and on the other hand that
\begin{equation}\label{reduc-1-d-2}
\sup_{n\geq 0}\, \mathbb{E} \bigg[  \Big| \int_0^1 |P|\big( x^{(n)}_u\big) \mathrm{d} x^{(n)}_u |Q|\big( x^{(n)}_u\big)\Big|^r \bigg] \ < \ \infty\ ,
\end{equation}
where $x^{(n)}$ stands for the linear interpolation of a (standard) Brownian motion $x$ along $(t_i^n)$.

\smallskip

Combining (\ref{fixed-d-1-brown})-(\ref{fixed-d-1-brown-strato}) with (\ref{unifo-conv-ito})-(\ref{unifo-conv-strato}), we can then switch the order of the limits just as in (\ref{limits-switch-start})-(\ref{limits-switch-end}), by considering the \enquote{Riemann-sum} approximation
$$\sum_{i=0}^{2^n-1}P\big( X^{(d)}_{t_i^n}\big) \der X^{(d)}_{t_i^n t_{i+1}^n} Q\big( X^{(d)}_{t_i^n}\big)$$
in the Itô case (\ref{conv-bro-ito}), and the \enquote{Wong-Zakaï} approximation
$$\int_0^1 P\big( X^{(d,n)}_u\big) \mathrm{d} X^{(d,n)}_u Q\big( X^{(d,n)}_u\big)$$
in the Stratonovich case (\ref{conv-bro-strato}).
\end{proof}

\subsection{About the extension of the convergence result to $H\in (\frac13,\frac12)$} 

\

\smallskip

When $H\in (\frac13,\frac12)$, and for any finite $d\geq 1$, we have seen that we can still define the integral $\int_0^1 P(X^{(d)}_u) dX^{(d)}_u Q(X^{(d)}_u)$ through the considerations of Section \ref{sec:integr-hfbm}, that is as the rough integral
$$\int_0^1 P(X^{(d)}_u) dX^{(d)}_u Q(X^{(d)}_u):=\int_0^1 (P(X^{(d)}_u) \otimes Q(X^{(d)}_u)) \sharp d\mathbb{X}^{(d)}_u \ ,$$
where $\mathbb{X}^{(d)}:=(X^d,\mathbf{X}^{(d)})$ and $\mathbf{X}^{(d)}$ is the product Lévy area derived from the canonical Lévy areas in (\ref{1d-levy-areas}). Besides, thanks to the continuity of the rough constructions, it can be shown that the so-defined integral satifies, for every $r\geq 1$,
\begin{equation}\label{fixed-d-rough}
\vp_d\Big( \Big( \int_0^1 P\big( X^{(d)}_u\big) \mathrm{d}X^{(d)}_u Q\big( X^{(d)}_u\big)\Big)^r\Big)=\lim_{n\to \infty} \vp_d\Big( \Big(  \int_0^1 P\big( X^{(d,n)}_u\big) \mathrm{d} X^{(d,n)}_u Q\big( X^{(d,n)}_u\big)\Big)^r\Big) \ ,
\end{equation}
where $X^{(d,n)}$ is the linear interpolation introduced in (\ref{linear-interpolation}), which thus extends the approximation property (\ref{fixed-d-1-brown-strato}).

\smallskip

These results happen to remain true for $d=\infty$, that is for a NC-fBm $X^{(\infty)}$ of Hurst index $H\in (\frac13,\frac12)$. Indeed, as stated in \cite[Proposition 2.9]{deya-schott-3}, we can also define the integral
$$\int_0^1 P(X^{(\infty)}_u) dX^{(\infty)}_u Q(X^{(\infty)}_u)$$
through some rough construction (using the \enquote{canonical} product Lévy area above $X^{(\infty)}$ exhibited in \cite[Proposition 2.8]{deya-schott-3}), and it holds that
\begin{equation}\label{approx-ncfbm-rough}
\vp_\infty\Big( \Big( \int_0^1 P\big( X^{(\infty)}_u\big) \mathrm{d}X^{(\infty)}_u Q\big( X^{(\infty)}_u\big)\Big)^r\Big)=\lim_{n\to \infty} \vp_\infty\Big( \Big(  \int_0^1 P\big( X^{(\infty,n)}_u\big) \mathrm{d} X^{(\infty,n)}_u Q\big( X^{(\infty,n)}_u\big)\Big)^r\Big) \ ,
\end{equation}
for every $r\geq 1$.

\

Combining (\ref{fixed-d-rough})-(\ref{approx-ncfbm-rough}) with the polynomial convergence of Proposition \ref{prop:convergence-result} (valid for every $H\in (0,1)$), the problem reduces, as before, to justifying the fact that we can switch the limits in $d$ and in $n$, along the same procedure as in (\ref{limits-switch-start})-(\ref{limits-switch-end}). Unfortunately, the arguments that we have used to this end in the proofs of Proposition \ref{prop:conv-integr} and Proposition \ref{prop:conv-integr-brown} (leading to a uniform-in-$n$ convergence as $d\to \infty$) are no longer valid when $H<\frac12$, since the (disjoint) increments of the fractional Brownian motion are then known to be negatively correlated. In other words, we can no longer ensure that the quantity $C_\pi$ in (\ref{lhs}) (or rather its counterpart when considering the approximation $\int_0^1 P\big( X^{(d,n)}_u\big) \mathrm{d} X^{(d,n)}_u Q\big( X^{(d,n)}_u\big)$) is always positive, which annihilates estimate (\ref{bound-c-pi}) and the possibility to go back to the 1d situation (i.e., to the consideration of the uniform estimate (\ref{reduc-1-d-2}), with $x$ a fBm of Hurst index $H\in (\frac13,\frac12)$).

\

In fact, using Lemmas \ref{lem:voicu}-\ref{lem:gene} and setting $\text{d}_P:=\text{deg}(P)$, $\text{d}_Q:=\text{deg}(Q)$, we can write the difference under consideration as
\begin{align*}
&\vp_d\Big( \Big( \int_0^1 P\big( X^{(d,n)}_u\big) \mathrm{d} X^{(d,n)}_u Q\big( X^{(d,n)}_u\big)\Big)^r\Big)- \vp_\infty\Big( \Big( \int_0^1 P\big( X^{(\infty,n)}_u\big) \mathrm{d} X^{(\infty,n)}_u Q\big( X^{(\infty,n)}_u\big)\Big)^r\Big)\\
&\hspace{2cm}=\sum_{g=1}^{\frac{r}{4}(\text{d}_P+\text{d}_Q)} d^{-2g} \vp^{(g)}_\infty\Big( \Big(\int_0^1 P\big( X^{(\infty,n)}_u\big) \mathrm{d} X^{(\infty,n)}_u Q\big( X^{(\infty,n)}_u\big)\Big)^r\Big) \ , 
\end{align*}
where for every fixed genus $g\geq 1$, the quantity $\vp^{(g)}_\infty\big( ...\big)$ is formally defined through the linear extension of the formula
$$\vp^{(g)}_\infty \big( X^{(\infty)}_{t_1}\cdots X^{(\infty)}_{t_r}\big)=\sum_{\substack{\pi \in \mathcal{P}_2(r)\\\text{genus}(\pi)=g}} \prod_{(p,q)\in \pi} c_H(t_p,t_q) \ .$$
Accordingly, for the desired uniform-in-$n$ convergence to be true (allowing to switch the limits in $d,n$), it would be sufficient to show that for all fixed genus $g\geq 1$ and order $r\geq 1$, 
$$\sup_{n\geq 0}\, \Big|  \vp^{(g)}_\infty\Big( \Big(\int_0^1 P\big( X^{(\infty,n)}_u\big) \mathrm{d} X^{(\infty,n)}_u Q\big( X^{(\infty,n)}_u\big)\Big)^r\Big) \Big| \ < \infty \  .$$

\smallskip

When $g=0$, this uniform estimate is a consequence of (\ref{approx-ncfbm-rough}), and thus follows from the (sophisticated) considerations of \cite{deya-schott-3}. We could then be tempted to try to extend the latter considerations to every $g\geq 1$ (starting from some kind of \enquote{NC-fBm with genus $g$} in a NC-probability space). Unfortunately, when doing so, one soon realizes that, contrary to $\vp^{(0)}_\infty$, the above functional $\vp_\infty^{(g)}$ (for $g\geq 1$) cannot be extended into a genuine positive trace, in the sense of Definition \ref{defi:nc-probability-space}, item $(2)$. For instance, noting that the sole pairing $\pi\in \mathcal{P}_2(4)$ with genus $1$ is the one given by $\{\{1,3\},\{2,4\}\}$, one has
\begin{align*}
&\vp^{(1)}_\infty\Big( \Big(X^{(\infty)}_1 X^{(\infty)}_2-X^{(\infty)}_2 X^{(\infty)}_1\Big)\Big(X^{(\infty)}_1 X^{(\infty)}_2-X^{(\infty)}_2 X^{(\infty)}_1\Big)^\ast\Big)\\
&=\vp^{(1)}_\infty\big( X^{(\infty)}_1 X^{(\infty)}_2 X^{(\infty)}_2 X^{(\infty)}_1\big)-\vp^{(1)}_\infty\big( X^{(\infty)}_1 X^{(\infty)}_2 X^{(\infty)}_1 X^{(\infty)}_2\big)\\
&\hspace{0.5cm}-\vp^{(1)}_\infty\big( X^{(\infty)}_2 X^{(\infty)}_1 X^{(\infty)}_2 X^{(\infty)}_1\big)+\vp^{(1)}_\infty\big( X^{(\infty)}_2 X^{(\infty)}_1 X^{(\infty)}_1 X^{(\infty)}_2 \big)\\
&=2 \{c_H(1,2)^2-c_H(1,1)c_H(2,2)\}=2\{2^{4H-2}-2^{2H}\}=2^{2H+1}\{2^{2H-2}-1\} \ < \ 0 \ .
\end{align*}
This observation immediately rules out the possibility to consider the non-commutative probability setting when $g\geq 1$ and so to adapt the developments of \cite{deya-schott-3}. 

\

As a result, when $H\in (\frac13,\frac12)$, the convergence property (understood in the sense of non-commutative probability)
$$\int_0^1 P\big( X^{(d)}_u\big) \mathrm{d}X^{(d)}_u Q\big( X^{(d)}_u\big)\stackrel{d\to \infty}{\longrightarrow} \int_0^1 P\big( X^{(\infty)}_u\big) \mathrm{d}X^{(\infty)}_u Q\big( X^{(\infty)}_u\big)$$
still remains a conjecture for the moment.

\end{document}